\documentclass[12pt,a4paper]{amsart}      

\usepackage{amsmath,amsthm,amssymb,latexsym,a4wide,tikz,multicol,tikz-cd}
\usepackage{arydshln,multirow}
\usepackage{tikz-qtree,tikz-qtree-compat}
\usepackage{mathtools,stmaryrd}
\usepackage{tikz}
\tikzset{font=\small}
\usepackage{enumerate}
\usetikzlibrary{matrix,arrows}
\usetikzlibrary{positioning}
\usetikzlibrary{cd}

\usepackage{mathrsfs}

\usepackage{graphicx}
\usepackage{amsmath,amssymb}
\usepackage{color,amscd,amsthm}
\usepackage{tikz}
\usepackage{float}
\usetikzlibrary{arrows, intersections}
\pgfdeclarelayer{bg}
\pgfdeclarelayer{fg}
\pgfdeclarelayer{crossing over}
\pgfsetlayers{main,bg,crossing over,fg}
\usepackage{hyperref}

\definecolor{brown}{rgb}{1.0, 0.66, 0.07}
\definecolor{green}{cmyk}{0.64,0,0.95,0.40}

\numberwithin{equation}{section}

\newtheorem{theorem}{Theorem}[section]
\newtheorem*{theorem*}{Theorem}
\newtheorem{lema}[theorem]{Lemma}
\newtheorem{lemma}[theorem]{Lemma}

\newtheorem{corol}[theorem]{Corollary}
\newtheorem{corollary}[theorem]{Corollary}
\newtheorem{prop}[theorem]{Proposition}
\newtheorem{proposition}[theorem]{Proposition}

\newtheorem{quest}{Question}

\newtheorem*{clm*}{Claim}

\newtheorem{definition}[theorem]{Definition}
\newtheorem*{definition*}{Definition}

\newcommand{\E}{\mathcal{E}}
\newcommand{\F}{\mathcal{F}}

\newcommand{\calA}{\mathcal{A}}
\newcommand{\calB}{\mathcal{B}}
\newcommand{\B}{\mathtt{wB}}

\newcommand{\calN}{\mathcal{N}}
\newcommand{\calC}{\mathcal{C}}

\newcommand{\NF}{\operatorname{NF}}

\newcommand{\set}[2]{\{#1:\ #2\}}

\DeclareMathOperator{\cl}{cl}
\DeclareMathOperator{\Int}{Int}

\newcommand{\uopen}[1]{\mathtt{Op}(#1)}
\newcommand{\uclosed}[1]{\mathtt{Cl}(#1)}
\newcommand{\uzero}[1]{\mathtt{Z}(#1)}
\newcommand{\ucozero}[1]{\mathtt{Cz}(#1)}

\newcommand{\w}{\omega}

\title{Ultrafilter approach to remote points}
\author{Serhii Bardyla}
\address{University of Vienna, Institute of Mathematics, Kolingasse 14-16, 1090 Vienna, Austria.}
\email{sbardyla@gmail.com}
\author{Jaroslav \v Supina}
 \address{Institute of Mathematics,  P.J. \v{S}af\'arik University in Ko\v sice,  Jesenn\'a 5, 040 01 Ko\v{s}ice, Slovakia}
\email{jaroslav.supina@upjs.sk}

\thanks{The research of the first named author was funded in whole by the Austrian Science Fund FWF [10.55776/ESP399]. Second named author was supported by the Slovak Research and Development Agency under the Contract no. APVV-20-0045 and by the~grant 1/0657/22 of Slovak Grant Agency VEGA}

\makeatletter


\title{New approaches to remote points}

\begin{document}
    
    \begin{abstract}
For a given Tychonoff space $X$, a point $p\in \beta(X)\setminus X$ is called {\em remote} if $p$ is not in the closure of any nowhere dense subset of $X$.    
In this paper, we characterize spaces with remote points in terms of certain topological ultrafilters, measures, and compact-like properties corresponding to the ideal consisting of nowhere dense sets. It is shown that the space of remote points is homeomorphic to a subspace of the Stone space taken over the smallest Boolean algebra containing all open and nowhere dense sets. Also, we show that the space of remote points of $\mathbb R$ is $\omega$-bounded. 
    \end{abstract}
\maketitle


\section{Introduction and main results}\label{S-intro}

A {\em filter} on a set $X$ is a family of subsets of $X$ that does not contain the empty set, and is closed under finite intersections and taking supersets. A filter $\F$ on a set $X$ is called an {\em ultrafilter} if for each $A\in X$ either $A\in \F$ or $X\setminus A\in \F$, or equivalently, $\F$ is maximal with respect to the inclusion among filters on $X$.
If we consider filters on a topological space $X$, it is natural to assume that they are somehow connected with the topological structure of $X$. An instance of such connection is presented in the following  definition. 
 
\begin{definition}\rm
Let $R$ be a property of subsets of a space $X$. A filter $\F$ on $X$ is called
\begin{enumerate}[\rm(i)]
\item an {\em $R$ filter} if $\F$ possesses a base consisting of sets with property $R$;
\item an {\em $R$ ultrafilter} if $\F$ is an $R$ filter, and for any subset $A\subseteq X$ with property $R$ either $A\in \F$ or $X\setminus A\in \F$;
\end{enumerate}
\end{definition}
It is straightforward to check that each $R$ ultrafilter is maximal with respect to the inclusion among $R$ filters. The converse implication holds true, if the property $R$ is preserved under finite intersections. 
In this paper, we are mostly dealing with the properties of being open, closed, zero, or cozero. Recall that a subset $A$ of a space $X$ is called a {\em zero set} if $A=f^{-1}(0)$ for some continuous function $f\colon X\rightarrow [0,1]$. Furthermore, the set $A$ is called {\em cozero} if $X\setminus A$ is a zero set.
If $\F$ is simultaneously an $R_i$~(ultra)filter for each $i\leq n$, then $\F$ is called an $R_0$-$\cdots$-$R_n$ (ultra)filter. 
For more on open and closed filters, see \cite{BSZ} and references therein.

For a Tychonoff space $X$ by $\beta (X)$ we denote the Stone-\v{C}ech compactification of $X$. We shall consider the representation of $\beta (X)$ given in~\cite[\S~2]{CoNe}. Namely, $\beta (X)$ is the set of all zero ultrafilters on $X$ endowed with the topology $\tau$ generated by the base
$$\mathcal B=\{\{u\in \beta (X): A\notin u\}: A\hbox{ is a zero subset of }X\}.$$ Note that in this representation, a point $x$ of $X$ is identified with a zero ultrafilter generated by all zero subsets of $X$ that contain $x$. As usual, the remainder $\beta (X)\setminus X$ is denoted by $X^*$. A~point~$p\in X^*$ is called {\em remote} if $p$ is not in the closure of any nowhere dense subset of~$X$. We write that a~space $X$ {\em has a~remote point} if there is a~remote point~$p\in X^*$. 
Remote points have been intensively studied in \cite{B08,BC09,ChS80,vD,ED81,vDvM84,D84,D89,DP88,HGHTM12, V91}. 
For applications of remote points see~\cite{AJW14,DHG19,Du,HG15}.
One of the central questions regarding remote points is the problem of their existence. 
For instance, the following result appears in~\cite{D89}.
\begin{theorem}[Dow]\label{pseudocompact}
    If $X$ is a~pseudocompact space, then $X$ does not have remote points.
\end{theorem}
 
Van Douwen~\cite{ED81} showed that each non-pseudocompact Tychonoff space with countable $\pi$-weight has a remote point.
The latter result was generalized in~\cite{D84} as follows.  
\begin{theorem}[Dow]\label{dow-02}
    If $X$ is a Tychonoff non-pseudocompact ccc space with $\pi$-weight at most $\omega_1$, then $X$ has a remote point.
\end{theorem}
Brown and Dow~\cite{BC09} have shown that, assuming CH, the cardinal $\omega_1$ in the latter theorem can be substituted by the cardinal $\omega_2$.  Fine and Gillman~\cite{FG} showed that under CH each separable non-pseudocompact normal space has remote points. On the other hand, Dow~\cite{D89} constructed a consistent example of a Tychonoff separable non-pseudocompact space with no remote points.

In this paper, we give several characterizations of spaces possessing a~remote point. Their formulation requires a few definitions, given below. By $o(X)$ and $cz(X)$ we denote the families of all open subsets of a space $X$ and all cozero subsets of $X$, respectively. 

\begin{definition}\rm
Let $J$ be an ideal on a space $X$. The space $X$ is called 
\begin{enumerate}[(i)]
\item {\em $J$-compact} if for each function $\phi\colon J\rightarrow o(X)$ such that $A\subseteq \phi(A)$ for all $A\in J$ there exists a finite family $\mathcal A\subseteq J$ such that $\bigcup_{A\in\mathcal A}\phi(A)= X$.
    \item {\em weakly $J$-compact} if for each function $\phi\colon J\rightarrow cz(X)$ such that for each $A\in J$ there exists a zero set $Z_A$ satisfying $A\subseteq Z_A\subseteq \phi(A)$ there exists a finite family $\mathcal A\subseteq J$ such that $\bigcup_{A\in\mathcal A}\phi(A)=X$. 
\end{enumerate}
\end{definition}


\begin{definition}\rm
 Let $u$ be an ultrafilter on a space $X$. Define
 \begin{enumerate}[\rm(i)]
\item $\uopen{u}$ is the filter on $X$ generated by all open elements of $u$;
\item $\uclosed{u}$ is the filter on $X$ generated by all closed elements of $u$;
\item $\uzero{u}$ is the filter on $X$ generated by all zero subsets of $X$ that are elements of $u$;
\item $\ucozero{u}$ is the filter on $X$ generated by all cozero subsets of $X$ that are elements of $u$.
\end{enumerate}   
\end{definition}


For a given space $X$ by $\NF$ we denote the ideal on $X$ consisting of all nowhere dense subsets of $X$ together with all finite subsets of $X$. A point $x$ is called an {\em accumulation point} of a filter $\F$, if $x\in \bigcap_{F\in \F}\overline{F}$.

The following result characterizes spaces with remote points.
\begin{theorem}\label{main1}
For a Tychonoff space $X$ the following conditions are equivalent:
\begin{enumerate}[\rm(a)]
    \item $X$ has a remote point;
    \item There exists a zero-cozero-closed-open ultrafilter on $X$ with no accumulation points;
    \item There exists a non-principal ultrafilter $u$ on $X$ such that $\uopen{u}=\uclosed{u}=\uzero{u}=\ucozero{u}$;
    \item $X$ is not weakly $\NF$-compact.
\end{enumerate}
\end{theorem}
 
Normal spaces with remote points can be additionally characterized by certain measures.
Let $X$ be a space and $\calA$ be an algebra of subsets of $X$, i.e., $\{\emptyset,X\}\subseteq \calA$ and $\calA$ is closed under finite unions, intersections and complements.
A function $\mu\colon\calA\rightarrow [0,\infty]$ is called a {\em finitely additive measure} if for every disjoint $A,B\in\calA$ we have $\mu(A\cup B)=\mu(A)+\mu(B)$. 
If, in addition, $\mu$ has values only $0$ or $1$, then $\mu$ is called a {\em finitely additive $\{0,1\}$-measure}.   
Following~\cite{buk-str}, a~measure $\mu$ on a~space $X$ is said to be {\em outer regular} if for each $A\in\calA$, $$\mu(A)=\inf\{\mu(U):\ A\subseteq U,\ U\text{ is open}\}.$$
Measures on topological spaces were investigated in \cite{BNR, BNP16,DP15, DP07}.

For a space $X$ let $$\B(X)=\{(U\setminus P)\cup Q\colon U \hbox{ is open in } X \hbox{ and } P,Q \hbox{ are nowhere dense in }X\}.$$
It is easy to check (see Lemma~\ref{algebraclear}) that $\B(X)$ is the smallest algebra containing all open and all nowhere dense subsets of $X$. 

The following result characterizes normal spaces with remote points. 
\begin{theorem}\label{wrB}
For a normal space $X$ the following conditions are equivalent:
\begin{enumerate}[\rm(a)]
\item $X$ has a remote point;
\item There is a closed-open ultrafilter on $X$ with no accumulation points;
\item There exists a non-principal ultrafilter $u$ on $X$ such that $\uopen{u}=\uclosed{u}$;
\item $X$ is not $\NF$-compact;
\item There is an outer regular finitely additive $\{0,1\}$-measure~$\mu$ defined on $\B(X)$ such that $\mu(A)=0$ for all $A\in\NF$;
\item There is an outer regular finitely additive measure~$\mu$ defined on $\B(X)$ such that $\mu(X)=\infty$ and $\mu(A)<\infty$ for all $A\in\NF$.   
\end{enumerate}    
\end{theorem}

For a Tychonoff space $X$, let $\rho(X)$ be the subspace of $X^*$ consisting of all remote points of $X$. The space $\rho(X)$ was investigated by Hernández-Gutiérrez, Hrušák and Tamariz-Mascarúa in~\cite{HGHTM12}. Note that for each space $X$ the algebra $\B(X)$ is a Boolean algebra.  By $S(\B(X))$ we denote the {\em Stone space} of the Boolean algebra $\B(X)$. For more about Stone spaces see~\cite[\S\ 2]{CoNe} or \cite{Kop89}. Recall that a space $X$ is called {\em $\omega$-bounded} if the closure of each countable subset of $X$ is compact. 

\begin{theorem}\label{Stone}
 For any Tychonoff space $X$, $\rho(X)$ is homeomorphic to a subspace of $S(\B(X))$.
 Moreover, if $X$ is additionally locally compact and $\sigma$-compact, then
  $\rho(X)$ is $\omega$-bounded.
 \end{theorem}


Theorems~\ref{main1},~\ref{wrB} and~\ref{Stone} are proven in Section~\ref{5} after all the preliminary work is done in sections \ref{s2}--\ref{s4}.

\section{Filters on topological spaces}\label{s2}


Let $R$ be a~property of subsets of a topological space $X$. We write that a subset $A\subseteq X$ has property $R^c$ if $X\setminus A$ has property $R$.   
\par
\begin{proposition}\label{complement_filters}
    Let $\F$ be a~filter on a space~$X$. If $\F$ is both, an $R$ filter and $R^c$ ultrafilter, then $\F$ is an $R$ ultrafilter. 
\end{proposition}
\begin{proof}
    Assume that an $R$ filter $\F$ is an $R^c$ ultrafilter and fix a set $A\not\in\F$ with property $R$. Since $\F$ is an $R^c$ ultrafilter, the set $X\setminus A$ is either an element of $\F$ or there is $F\in\F$ disjoint with $X\setminus A$. The latter case yields $F\subseteq A$, and so $A\in\F$, a~contradiction. Hence the set $X\setminus A$ is an element of $\F$ disjoint with~$A$. Thus $\F$ is an $R$ ultrafilter.
\end{proof}
\par

Proposition~\ref{complement_filters} implies the following.

\begin{corollary}\label{corollary list}
Let $\F$ be a~filter on a space~$X$. Then the following hold: 
\begin{enumerate}[\rm(1)]
    \item If $\F$ is an open filter and a closed ultrafilter, then $\F$ is an open ultrafilter.
    \item If $\F$ is a~closed filter and an open ultrafilter, then $\F$ is a~closed ultrafilter.
    \item If $\F$ is a~cozero filter and a~zero ultrafilter, then $\F$ is a~cozero ultrafilter.
    \item If $\F$ is a~zero filter and a~cozero ultrafilter, then $\F$ is a~zero ultrafilter.
    \item If $\F$ is an open-closed filter, then $\F$ is a~closed ultrafilter if and only if $\F$ is an open ultrafilter.    
    \item If $\F$ is a~zero-cozero filter, then $\F$ is a~cozero ultrafilter if and only if $\F$ is a~zero ultrafilter.    
\end{enumerate}   
\end{corollary}
In addition to item (1) in Corollary \ref{corollary list}, we give an alternative assumption that forces a closed ultrafilter to be an open ultrafilter.
\begin{prop}\label{op-to-cl_ultra}
Let $\F$ be a closed ultrafilter on a space $X$ such that $A\notin \F$ for each nowhere dense subset $A$ of $X$. Then $\F$ is an open ultrafilter. 
\end{prop}
\begin{proof}
Fix any $F\in\F$. Since $\F$ is a closed filter, there exists a~closed $F'\in \F$ such that $F'\subseteq F$.
 Note that the set $P:=F'\setminus \Int(F')$ is nowhere dense, so $P\not\in\F$. Since the set $P$ is closed and $\F$ is a closed ultrafilter, there is $G\in\F$ such that $G\cap P=\emptyset$. Then $G\cap F'\subseteq \Int(F')$, implying that $\Int(F')\in\F$. Since $F$ was arbitrarily chosen, $\F$ is an open filter. Corollary~\ref{corollary list}(i) implies that $\F$ is an open ultrafilter. 
\end{proof}
Recall that a space $X$ is said to have {\em ccc property}, if each family of pairwise disjoint open sets is at most countable. A subset $A$ of a space $X$ is said to have ccc property, if $A$ equipped with the subspace topology has ccc property.

\begin{prop}\label{cl-zero}
Let $\F$ be an open ultrafilter on a Tychonoff space $X$ such that $\F$ has an element with a ccc property.
    Then $\F$ is a~cozero ultrafilter.       
\end{prop}
\begin{proof}
It suffices to show that $\F$ has a~base consisting of cozero sets. Fix an arbitrary set $F\in\F$ with ccc property. It is straightforward to check that each open subset of $F$ has ccc property. Then the family $\mathcal B=\{U\subseteq F: U\hbox{ is open and } U\in\F\}$ forms a base of $\F$ consisting of open sets with ccc property. Fix any $B\in\mathcal B$.  Find a~maximal pairwise disjoint family $\calC$ of subsets of $B$, which are cozero in~$X$. By ccc property of $B$, the family $\calC$ is countable. Since $X$ is Tychonoff, the maximality of $\calC$ implies that the open subset $\bigcup\calC$ is dense in $B$. Lemma~3.1 from \cite{BSZ} implies that $\bigcup\calC\in \F$. Moreover, $\bigcup\calC$, being a~countable union of cozero sets, is a~cozero set itself. Hence $\F$ is a cozero filter, as required.
\end{proof}


\begin{prop}\label{pnormal}
 If $\F$ is a closed-open ultrafilter on a~normal space $X$, then $\F$ is a~zero-cozero ultrafilter.      
 \end{prop}

\begin{proof}
It suffices to show that $\F$ is a zero-cozero filter. Fix any set $F\in \F$. Since $\F$ is an open filter, without loss of generality we can assume that $F$ is open. 
Since $\F$ is a closed filter, there exists a closed $G\subseteq F$ such that $G\in\F$. Since the space $X$ is normal, by the Urysohn Lemma~\cite[Theorem 1.5.11]{Eng} there exists a continuous function $f\colon X\rightarrow [0,1]$ such that $f{\restriction}_G\equiv 1$ and $f{\restriction}_{X\setminus F}\equiv 0$. Observe that $G\subseteq f^{-1}([1/2,1])\subseteq f^{-1}((0,1])\subseteq F$, implying $f^{-1}((0,1])\in \F$ and $f^{-1}([1/2,1])\in\F$. Since $f^{-1}((0,1])$ is a cozero set and $f^{-1}([1/2,1])$ is a zero set, we get that $\F$ is a zero-cozero filter, as required.    
\end{proof}

The following question is left open.

\begin{quest}\label{quest}
Does there exist a Tychonoff space $X$ possessing a free closed-open ultrafilter that is not a zero-cozero ultrafilter?
\end{quest}

Noteworthy, the filter on the ordinal $\w_1$ (endowed with the order topology) generated by complements of all initial segments is a zero-cozero ultrafilter that is neither open ultrafilter nor closed ultrafilter.

Recall that a space $X$ has a {\em large inductive dimension $0$} if two nonempty disjoint closed subsets can be separated by a clopen set. The following result gives a partial answer to Question~\ref{quest}.

\begin{proposition}
Assume that a space $X$ contains a dense open subspace of large inductive dimension $0$. If $\F$ is a closed-open ultrafilter on $X$ then $\F$ is a zero-cozero ultrafilter.    
\end{proposition}

\begin{proof}
Fix any element $W\in\F$. It suffices to find a clopen set $M\in \F$ such that $M\subseteq W$, as clopen sets are both zero and cozero. Let $D$ be a dense open subspace of $X$ with large inductive dimension $0$. Since $\F$ is an open ultrafilter, $D\in \F$. Since $\F$ is a closed-open filter, there exist an open set $V\in\F$ and a closed set $F\in \F$ such that $$F\subseteq V\subseteq \cl_X(V)\subseteq D\cap W.$$ 
Since $D$ has large inductive dimension $0$, there exists a clopen subset $M$ of $D$ such that $F\subseteq M\subseteq V$. Since $D$ is open in $X$, we get that $M$ is open in $X$ too. Since $M\subseteq \cl_X(V)\subseteq D$, we get that $M$ is closed in $X$. Hence $M$ is a clopen element of $\F$, as required.     
\end{proof}


A family $\mathcal A$ of subsets of a space $X$ is called {\em locally finite} if for each $x\in X$ there exists an open neighborhood of $x$ that intersects only finitely many members of $\mathcal A$.  
Recall that a space $X$ is called {\em feebly compact} if each locally finite family of nonempty open sets is finite. By~\cite[Theorem 3.10.22]{Eng}, a Tychonoff space is feebly compact if and only if it is pseudocompact. 
Taking into account Theorem \ref{main1}, the following proposition implies Theorem~\ref{pseudocompact}.

\begin{proposition}\label{feebly compact}
Let $\F$ be a closed-open ultrafilter on a space $X$ with no accumulation points. Then $X$ is not feebly compact.     
\end{proposition}

\begin{proof}
By Zorn's Lemma, there exists a maximal family $\mathcal A$ consisting of open pairwise disjoint subsets of $X$ that are not in $\F$. Since the filter $\F$ has no accumulation points, the set $\bigcup\mathcal A$ is dense in $X$. As $\F$ is an open ultrafilter, $\bigcup\mathcal A\in\F$. Since $\F$ is a closed-open filter there exists an open set $F\in\F$ such that $\overline{F}\subseteq \bigcup\mathcal A$. We claim that the family $\mathcal A'=\{F\cap A: A\in\mathcal A\}$ is infinite and locally finite. Seeking a contradiction, assume that $\mathcal A'$ is finite. Then there exists a finite subfamily $\mathcal B$ of $\mathcal A$ such that $F\subseteq \bigcup \mathcal B$. Since $\F$ is an open ultrafilter and elements of $\mathcal B$ are pairwise disjoint, there exists $B\in \mathcal B\subseteq\mathcal A$ such that $B\in \F$, which contradicts the choice of $\mathcal A$. Hence $\mathcal A'$ is infinite. To show that $\mathcal A'$ is locally finite, fix any $x\in X$. If $x\notin\bigcup \mathcal A$, then $X\setminus \overline{F}$ is an open neighborhood of $x$ disjoint with $\bigcup \mathcal A'$. If $x\in \bigcup \mathcal A$, then there exists a unique $A\in \mathcal A$ that contains $x$. Note that $A$ is an open neighborhood of $x$ intersecting at most one element of $\mathcal A'$.   Hence $\mathcal A'$ is an infinite locally finite collection of nonempty open subsets, witnessing that $X$ is not feebly compact. 
\end{proof}

The following result implies that in some cases feeble compactness is equivalent to non-existence of a closed-open ultrafilter without accumulation points, and as such complements Proposition~\ref{feebly compact}.  

\begin{proposition}
Let $X$ be a space with a dense discrete subspace. Then the following assertions are equivalent:
\begin{enumerate}[\rm(a)]
    \item $X$ contains an infinite closed set consisting of isolated points;
    \item $X$ possesses a closed-open ultrafilter with no accumulation points.
    \end{enumerate}
\end{proposition}

\begin{proof}
(a)$\Rightarrow$(b) Let $D$ be an infinite closed set consisting of isolated points. Then each nonprincipal ultrafilter $u$ on $D$ generates a closed-open ultrafilter on $X$ with no accumulation points. 


(b)$\Rightarrow$(a) Let $\F$ be a closed-open ultrafilter on $X$ with no accumulation points and $D$ be a dense discrete subset of $X$. It follows that $D$ consists of isolated points and thus is open. Hence $D\in \F$. Since $\F$ is a closed filter, there exists a closed set $F\in \F$ such that $F\subseteq D$. Then the closed set $F$ consists of isolated points. Since the filter $\F$ has no accumulation points, the set $F$ is infinite.
\end{proof}

Further we investigate a relation between the usual set-theoretic ultrafilters on a given space and the ``topological'' ultrafilters discussed above. 
 Let $u$ be an ultrafilter on a space $X$. For a property $R$ of subsets of $X$, let $R(u)$ be the $R$ filter on $X$ generated by the elements of $u$ with property $R$. The filters of the form $R(u)$ are called prime. For more information about prime filters we refer the reader to \cite{C85, C87,C89,ZF72,M95,M98}.
The next proposition shows that for each ultrafilter $u$ on a space $X$, the filter $R(u)$ is somewhat close to being an $R$ ultrafilter. 
The~equivalence of (a) and (b) in the~next theorem has been shown by Frol\'ik~\cite{ZF72} for closed filters.
\begin{proposition}\label{prime-char}
    Let $\F$ be an~$R$~filter on a space $X$. Then the following are equivalent:
    \begin{enumerate}[\rm(a)]
   \item Let $H_1,\ldots H_n$ be subsets of $X$ with property $R$ such that $\bigcup_{i\leq n}H_i\in\F$. Then $H_i\in\F$ for some $i\leq n$;
    \item There is an~ultrafilter~$u$ on~$X$ such that $\F=R(u)$;
    \item Let $\F_1,\cdots,\F_n$ be $R$~filters, where $\F_i\neq \F$ for all $i\leq n$. Then $\F\neq \bigcap_{i\leq n}\F_i$.
    \end{enumerate}
\end{proposition}
\begin{proof}
    (a)$\Rightarrow$(b)  
    Let $$\E=\{E\subseteq X: E\notin\F, E \hbox{ has property }R \hbox{ and } E\cap F\neq \emptyset \hbox{ for all }F\in\F\}.$$ 
    Let us show that the~family $\Theta=\F\cup\set{X\setminus E}{E\in\E}$ has the finite intersection property.
      Seeking a contradiction, assume that there exist sets $E_i\in\E$, $i\leq n$ and $F\in\F$ satisfying $$\bigcap_{i\leq n}(X\setminus E_i)\cap F=(X\setminus \bigcup_{i\leq n}E_i)\cap F=\emptyset.$$ Thus $F\subseteq\bigcup_{i\leq n}E_i$, and so $\bigcup_{i\leq n}E_i\in\F$. By the assumption there is $i\leq n$ such that $E_i\in\F$. But this contradicts the equality $\E\cap\F=\emptyset$. Hence the~family $\Theta$ has the~finite intersection property. Pick any ultrafilter $u$ that contains $\Theta$. Since $\F\subseteq \Theta\subseteq u$ and $\F$ is an $R$ filter, we obtain $\F\subseteq R(u)$. In order to show the converse inclusion, fix any element $C\in R(u)$ with property $R$. Observe that $C\notin \E$, as otherwise, $X\setminus C\in\Theta\subseteq u$, which would yield a contradiction. 
     On the other hand, since $C\in u$ and $\F\subseteq u$, we get that $C\cap F\neq \emptyset$ for each $F\in\F$. The definition of $\E$ implies that $C\in\F$. Thus $\F=R(u)$, as required.
    \par
    (b)$\Rightarrow$(c) Seeking a~contradiction, assume that there exist $R$ filters $\F_i$, $i\leq n$ such that all of them are distinct from $\F$, and $R(u)=\F=\bigcap_{i\leq n}\F_i$. Then for each $i\leq n$ there exist a subset $A_i\in\F_i\setminus R(u)$ with property $R$. It follows that $A_i\notin u$ for all $i\leq n$. Observe that 
    $$\bigcup_{i\leq n} A_i\in\bigcap_{i\leq n}\F_i=R(u)\subseteq u.$$ Since $u$ is an ultrafilter, there exists $i\leq n$ such that $A_i\in u$, a contradiction. 
    \par

   We prove (c)$\Rightarrow$(a) by contraposition. That is we actually prove $\neg$(a)$\Rightarrow\neg$(c). Suppose that $\F$ does not satisfy (a). Then there exist subsets $A_i$, $i\leq n$ of $X$ with property $R$ such that $\bigcup_{i\leq n}A_i\in\F$ and $A_i\notin \F$ for all $i\leq n$. Let $M=\{i\leq n: A_i\cap F\neq \emptyset \hbox{ for all } F\in\F\}$. It is easy to check that $\bigcup_{i\in M}A_i\in \F$. For each $i\in M$ let $\F_i$ be the $R$ filter generated by $\F\cup\{A_i\}$. The sets $A_i$, $i\in M$ witness that the filters $\F_i$, $i\in M$ are distinct from~$\F$. Observe that the filter $\bigcap_{i\in M}\F_i$ is generated by the family 
   $\mathcal B=\{F\cap\bigcup_{i\in M}A_i: F\in\F\}$,
which is a base for $\F$. Hence $\F=\bigcap_{i\in M}\F_i$, as required.
\end{proof}
In the following we explore conditions guaranteeing that the filter $R(u)$ is an ultrafilter in a~topological sense. Recall that for a given property $R$ a subset $A$ of a space $X$ is said to have property $R^c$ if $X\setminus A$ has property $R$.

\begin{proposition}\label{R(u)}
Let $u$ be an ultrafilter on a space $X$. 
Then $R(u)$ is an $R$-$R^c$ ultrafilter if and only if $R(u)=R^c(u)$. 
\end{proposition}

\begin{proof}
($\Rightarrow$) Since $R(u)$ has a base consisting of elements of $u$ with property $R^c$, we get $R(u)\subseteq R^c(u)$. Since $R(u)$ is an $R^c$ ultrafilter, it is easy to check that $R(u)$ is maximal with respect to the inclusion among $R^c$~filters. Hence $R(u)=R^c(u)$, as required.

($\Leftarrow$) Since $R(u)=R^c(u)$, we get that $R(u)$ is an $R$-$R^c$ filter. Fix a subset $A\subseteq X$ with property $R$ such that $A\notin R(u)$. Then $A\notin u$. Since $u$ is an ultrafilter, we get $X\setminus A\in u$. Then $X\setminus A\in R^c(u)$, as the set $X\setminus A$ has property $R^c$. Since $R(u)=R^c(u)$, we get that $X\setminus A\in R(u)$. Hence $R(u)$ is an $R$ ultrafilter. One can symmetrically show that $R(u)$ is an $R^c$ ultrafilter.   
\end{proof}




\begin{corollary}\label{blabla}
Let $u$ be an ultrafilter on $X$. Then the following holds:
\begin{enumerate}[\rm(1)]
\item $\uopen{u}$ is a closed-open ultrafilter if and only if $\uopen{u}=\uclosed{u}$.
\item $\uzero{u}$ is a zero-cozero ultrafilter if and only if $\uzero{u}=\ucozero{u}$.
\item $\uopen{u}$ is a closed-open-zero-cozero ultrafilter if and only if $\uopen{u}=\uclosed{u}=\uzero{u}=\ucozero{u}$. 
\end{enumerate}
\end{corollary}

\begin{proof}
Assertions (1) and (2) are immediate consequences of Proposition~\ref{R(u)}.

(3) ($\Leftarrow$) Items (1) and (2) imply that $\uopen{u}$ is a closed-open ultrafilter, and  $\uzero{u}$ is a zero-cozero ultrafilter. By the assumption, $\uopen{u}=\uzero{u}$. Hence $\uopen{u}$ is a closed-open-zero-cozero ultrafilter.

($\Rightarrow$) Assertion (1) implies that $\uopen{u}=\uclosed{u}$. Since each cozero set is open and each zero set is closed, we get that $\ucozero{u}\subseteq \uopen{u}$ and $\uzero{u}\subseteq \uclosed{u}$. Since $\uopen{u}$ is a cozero filter, $\ucozero{u}=\uopen{u}$. It remains to show that $\uclosed{u}\subseteq \uzero{u}$. Since $\uopen{u}=\uclosed{u}$ it suffices to show that $\uopen{u}\subseteq \uzero{u}$. By the assumption,  $\uopen{u}$ has a base consisting of zero sets which are elements of the ultrafilter $u$.  
Thus $\uopen{u}\subseteq \uzero{u}$, and $\uopen{u}=\uclosed{u}=\uzero{u}=\ucozero{u}$, as required.
\end{proof}

\begin{lemma}\label{trivial1}
Let $u$ be a non-principal ultrafilter on a $T_1$ space $X$. If $\uopen{u}$ is a closed-open filter, then $\uopen{u}$ has no accumulation points.
\end{lemma}

\begin{proof}
 Fix any $x\in X$. Since $u$ is not principal, we get that $A=X\setminus\{x\}\in u$. Since the space $X$ is $T_1$, $A\in\uopen{u}$. As $\uopen{u}$ is a closed filter, there exists a closed set $B\in \uopen{u}$ such that $B\subseteq A$. It is clear that the open neighborhood $X\setminus B$ of $x$ witnesses that $x$ is not an accumulation point of $\uopen{u}$. Since the point $x$ was chosen arbitrarily, the filter $\uopen{u}$ has no accumulation points.     
 \end{proof}

\begin{lemma}\label{trivial2}
Let $\F$ be an $R$ ultrafilter on $X$. Then for every ultrafilter $u$ that contains $\F$, we have $R(u)=\F$.
\end{lemma}

\begin{proof}
It is clear that $\F\subseteq R(u)$. It is easy to see that $\F$ is maximal with respect to the inclusion among $R$ filters. Thus $R(u)=\F$, as required.
\end{proof}

\section{Measures on topological spaces}

In this section, we show how the topological filters investigated in the previous section are connected to certain measures. 
Recall that $\B(X)$ denotes the family 
$$\{(U\setminus P)\cup Q\colon U \hbox{ is open in } X \hbox{ and } P,Q \hbox{ are nowhere dense in }X\}.$$

The following lemma is folklore. Nevertheless we include its short proof.

\begin{lema}[Folklore]\label{algebraclear}
For every space $X$, $\B(X)$ is an algebra of subsets of $X$. Moreover,  $\B(X)$ is the smallest algebra containing all open and nowhere dense subsets of $X$. 
\end{lema}

\begin{proof}
 It is clear that $\{\emptyset, X\}\subset \B(X)$. Fix any $A,B\in\B(X)$. Then there exist open sets $U_A, U_B$ and nowhere dense sets $P_A,Q_A, P_B,Q_B$ such that $A=(U_A\setminus P_A)\cup Q_A$ and $B=(U_B\setminus P_B)\cup Q_B$. 
 Then the sets $$P:=(P_A\setminus B)\cup(P_B\setminus A) \quad\hbox{
 and }\quad Q:=Q_A\cup Q_B$$ are nowhere dense, and $A\cup B=((U_A\cup U_B)\setminus P)\cup Q\in\B(X)$. Since
 \[
 X\setminus A=\big(\Int(X\setminus U_A)\setminus Q_A\big)\cup\big(((X\setminus U_A)\setminus\Int(X\setminus U_A))\setminus Q_A\big)\cup(P_A\setminus Q_A),
 \]
and $((X\setminus U_A)\setminus\Int(X\setminus U_A))\setminus Q_A\cup(P_A\setminus Q_A)$ is a nowhere dense subset of $X$, we get $X\setminus A\in\B(X)$.

Consider an algebra $\calB$ containing all open and nowhere dense subsets of $X$. Then $\calB$ contains all sets of the form $(U\setminus P)\cup Q$ for $U$ open and $P,Q$ nowhere dense. Hence $\B(X)\subseteq\calB$, implying that $\B(X)$ is the smallest algebra containing all open and nowhere dense subsets of $X$.
\end{proof}

Let $\mathcal A$ be an algebra of subsets of a space $X$. A subset $B\subseteq X$ is said to have {\em property $\mathcal A$} if $B\in \mathcal A$.

\begin{lema}
Let $X$ be a space. Then $\F$ is a closed-open ultrafilter on $X$ if and only if $\F$ is a $\B(X)$ ultrafilter such that for every nowhere dense subset $N$ of $X$ there exists an open set $V\supseteq N$ with $V\notin \F$.     
\end{lema}

\begin{proof}\label{wB(X) ultrafilter}
($\Rightarrow$)   
As $\B(X)$ contains all open (and closed) sets, we get that $\F$ is a $\B(X)$ filter. Fix any subset $A\in\B(X)$ such that $A\notin \F$. There exist an open set $U_A$ and nowhere dense sets $P_A, Q_A$ such that $A=(U_A\setminus P_A)\cup Q_A$. 
Since $\F$ is an open ultrafilter, the open dense set $D:=X\setminus (\overline{P_A}\cup \overline{Q_A})$ is in $\F$. It follows that $U_A\notin \F$, as otherwise the inclusion $A\supseteq D\cap U_A\in\F$ yields a contradiction. As $\F$ is an open ultrafilter, $X\setminus U_A\in\F$. Finally, $(X\setminus U_A)\cap D$ is an element of $\F$ disjoint with $A$. Hence $\F$ is a $\B(X)$ ultrafilter. Fix any nowhere dense subset $N$ of $X$. Since $\F$ is an open ultrafilter, the open dense set $X\setminus \overline{N}$ belongs to $\F$. Since $\F$ is a closed ultrafilter, there exists a closed set $C\in\F$ such that $C\subseteq X\setminus \overline{N}$. Then $V=X\setminus C$ is an open set that contains $N$ and does not belong to $\F$.

($\Leftarrow$) Since $\B(X)$ contains all open (and closed) sets, it suffices to show that $\F$ is a closed-open filter. Fix any $A\in \B(X)\cap \F$. There exist an open set $U_A$ and nowhere dense sets $P_A, Q_A$ such that $A=(U_A\setminus P_A)\cup Q_A$. By the assumption, each element of $\F$ is dense in an open set. Since $Q_A\notin \F$, $Q_A\in\B(X)$, and $\F$ is a $\B(X)$ ultrafilter, $X\setminus Q_A\in\F$. It follows that $U_A\supseteq A\cap (X\setminus Q_A)\in\F$. By the assumption, $\overline{P_A}\in \B(X)\setminus \F$. Since $\F$ is a $\B(X)$ ultrafilter, we get $X\setminus \overline{P_A}\in\F$ and, consequently, $B:=U_A\cap (X\setminus \overline{P_A})\in\F$. As $B$ is an open set contained in $A$ we get that $\F$ is an open filter. Consider the nowhere dense set $N:=\overline{B}\setminus B$. By the assumption, there exists an open set $V$ such that $N\subseteq V$ and $V\notin\F$. Since $V\in\B(X)$ and $\F$ is a $\B(X)$ ultrafilter, $X\setminus V\in\F$. It is easy to check that $C=B\cap (X\setminus V)\in \F$ is a closed subset contained in $B$. Hence $\F$ is also a closed filter, as required. 
\end{proof}

Following~\cite{buk-str}, a~measure $\mu$ on a~space $X$ is said to be {\em non-atomic} if $\mu(\{x\})=0$ for each $x\in X$.
For a $\{0,1\}$-measure $\mu$ let $\F_\mu$ be the filter whose base consists of measure one sets.

\begin{prop}\label{measuretech}
    Let $\mu$ be a finitely additive $\{0,1\}$-measure defined on an algebra $\calA$ of subsets of a space $X$. Then the following assertions hold:
\begin{enumerate}[\rm(1)]
    \item $\F_\mu$ is an $\calA$ ultrafilter.
    \item If $\mu$ is outer regular, then $\F_\mu$ is a~closed filter.
    \item If $\B(X)\subseteq \calA$ and $\mu$ is outer regular, then $\F_\mu$ is a~closed ultrafilter.
    \item If $\mu$ is outer regular and non-atomic, then $\F_\mu$ has no accumulation points.    
     \item If $\B(X)\subseteq \calA$, $\mu$ is outer regular and vanishes on nowhere dense sets, then $\F_\mu$ is a~closed-open ultrafilter. 
     \item If $X$ is $T_1$, $\B(X)\subseteq \calA$, $\mu$ is outer regular and vanishes on nowhere dense sets, then $\F_\mu$ has no accumulation points if and only if $\{x\}\notin\F_\mu$ for any isolated point $x\in X$.       
\end{enumerate}        
\end{prop}
\begin{proof}
(1) By the definition, $\F_\mu$ is an $\calA$ filter. Fix any $A\in\calA$ such that $A\notin \F_\mu$. Then $\mu(A)=0$. Since $\mu$ is finitely additive and $\mu(X)=1$, we get that $\mu(X\setminus A)=1$. Thus, $X\setminus A\in\F_\mu$ witnessing that $\F_\mu$ is an $\calA$ ultrafilter.

(2) Fix any $F\in\F_\mu$. By the outer regularity of $\mu$, there exists an open set $C$ such that $X\setminus F\subseteq C$ and $\mu(C)=0$. It follows that $X\setminus C\in \F_\mu$ is a closed subset of $F$, witnessing that $\F_\mu$ is a closed filter.

Assertion (3) follows from assertions (1) and (2).

(4) Fix any $x\in X$. Since $\mu(\{x\})=0$, the outer regularity of $\mu$ yields an open neighborhood $U$ of $x$ such that $\mu(U)=0$. It follows that $X\setminus U\in\F_\mu$ and, as such, the filter $\F_\mu$ has no accumulation points.

(5) By assertion (3), $\F_\mu$ is a closed ultrafilter. Proposition~\ref{op-to-cl_ultra} implies that $\F_\mu$ is an open ultrafilter. 


(6) The implication ($\Rightarrow$) is obvious. 

($\Leftarrow$) Assertion (5) implies that $\F_\mu$ is a closed-open ultrafilter. Fix any $x\in X$. If $x$ is isolated, then by assumption $\{x\}\notin \F_\mu$. Since $\F_\mu$ is an open ultrafilter, $X\setminus\{x\}\in \F_\mu$. Then the open set $\{x\}$ witnesses that $x$ is not an accumulation point of $\F_\mu$. Suppose that the point $x$ is not isolated. Since $\{x\}$ is a nowhere dense set, we get that $\{x\}\notin \F_\mu$. Since $\F_\mu$ is a closed ultrafilter and $\{x\}$ is a closed set, there exists a closed set $C\in\F_\mu$ such that $x\notin C$. Then the open set $X\setminus C$ witnesses that $x$ is not an accumulation point of $\F_\mu$. Since the point $x$ is chosen arbitrarily, the filter $\F_\mu$ has no accumulation points.  
\end{proof}

For any closed-open ultrafilter $\F$ on a space $X$ define the map $\mu_\F\colon \B(X)\rightarrow \{0,1\}$ by $\mu_\F(A)=1$ if $A\in \F$ and $\mu_\F(A)=0$, otherwise. 

\begin{prop}\label{regularmeasure}
Let $\F$ be a closed-open ultrafilter on a given space $X$. Then $\mu_\F$ is an outer regular $\{0,1\}$-measure that is defined on $\B(X)$ and vanishes on nowhere dense sets.     
\end{prop}

\begin{proof}
In order to show that $\mu_\F$ is a measure, fix any disjoint sets $A,B\in\B(X)$. Since $\F$ is a filter and $A\cap B=\emptyset$, either $A\notin\F$ or $B\notin\F$. If $A\in\F$ and $B\notin\F$ or $B\in\F$ and $A\notin\F$, we get $\mu_\F(A\cup B)=1=\mu_\F(A)+\mu_\F(B)$. It remains to consider the case when $A\notin\F$ and $B\notin\F$. By the definition of $\B(X)$, there exist open sets $U_A, U_B$ and nowhere dense sets $P_A,Q_A,P_B,Q_B$ such that $A=(U_A\setminus P_A)\cup Q_A$ and $B=(U_B\setminus P_B)\cup Q_B$. Since $\F$ is an open ultrafilter, the open dense set $X\setminus \overline{P_A}$ belongs to $\F$. We claim that $U_A\notin \F$. Indeed, assuming the contrary, we get $A\supseteq (X\setminus P_A)\cap U_A\in\F$, which contradicts the choice of $A$. Thus $U_A\notin \F$. Similarly, it can be shown that $U_B\notin \F$. Since $\F$ is an open ultrafilter,  $\{X\setminus U_A, X\setminus \overline{Q_A}, X\setminus U_B, X\setminus \overline{Q_B}\}\subseteq \F$. It follows that 
$$X\setminus (A\cup B)\supseteq (X\setminus U_A)\cap (X\setminus \overline{Q_A})\cap (X\setminus U_B)\cap (X\setminus \overline{Q_B}) \in \F.$$
Thus $\mu_\F(A\cup B)=0=\mu_{\F}(A)+\mu_\F(B)$. Hence $\mu_\F$ is a $\{0,1\}$-measure defined on $\B(X)$.

Since each element of $\F$ has nonempty interior, $\mu_\F$ vanishes on nowhere dense sets. It remains to show the outer regularity of $\mu_\F$. Fix any $A\in\B(X)$ such that $\mu_\F(A)=0$. Then $X\setminus A\in\F$. Since $\F$ is a closed filter, there exists a closed set $C\in\F$ such that $C\subseteq X\setminus A$. Then $\mu_\F(X\setminus C)=0=\mu_\F(A)$, witnessing that the measure $\mu_\F$ is outer regular.  
\end{proof}

The following corollary of Propositions~\ref{measuretech} and~\ref{regularmeasure}  establishes a one-to-one correspondence between closed-open ultrafilters on a space $X$ with no accumulation points and outer regular $\{0,1\}$-measures which are defined on $\B(X)$ and vanish on nowhere dense and finite sets. 

\begin{corollary}\label{correspondence}
For a space $X$ the following assertions hold:
\begin{enumerate}[\rm(i)]
   \item If $\mu$ is an outer regular $\{0,1\}$-measure that is defined on $\B(X)$ and vanishes on nowhere dense and finite sets, then $\F_\mu$ is a closed-open ultrafilter without accumulation points.
     \item If $\F$ is a closed-open ultrafilter on $X$ with no accumulation points, then $\mu_\F$ is an outer regular $\{0,1\}$-measure that is defined on $\B(X)$ and vanishes on nowhere dense and finite sets.
\end{enumerate}  
\end{corollary}


\section{Remote points}\label{s4}

For a point $y$ of a space $X$ by $\mathcal N(y)$ we denote the open filter on $X$ generated by the set of all open neighborhoods of $y$. The following was proven in~\cite[Theorem 14.2]{ED81}.
\par
\begin{theorem}[van Douwen]\label{douwen-traces}
    Let $X$ be a Tychonoff space and $p\in X^*$. Then the following conditions are equivalent.
    \begin{enumerate}[\rm(a)]
    \item $p$ is a~remote point.
    \item The trace of~$\calN(p)$ on~$X$ is an open ultrafilter.
    \item The trace of~$\calN(p)$ on~$X$ is a~closed ultrafilter.
    \end{enumerate}
\end{theorem}
\par


The following theorem reveals a correspondence between zero-cozero-closed-open ultrafilters without  accumulation points on a space~$X$ and remote points from $X^*$. 

\begin{proposition}\label{remote}
Let $X$ be a Tychonoff space. Then $\F\in X^*$ is a remote point if and only if $\F$ is a zero-cozero-closed-open ultrafilter on $X$ with no accumulation points.
\end{proposition}

\begin{proof}
    ($\Rightarrow$): Recall that elements of $X^*$ are zero ultrafilters. Assume that a zero ultrafilter $\F\in X^\ast$ is a~remote point and denote the trace of $\calN(\F)$ on~$X$ by $\mathcal H$. Since $\beta(X)$ is Hausdorff, we get that $\F$ has no accumulation points in $X$. By Theorem~\ref{douwen-traces}, the filter $\mathcal H$ is a~closed-open ultrafilter. Fix any $U\in\mathcal N(\F)$.
    Let us first show that $\mathcal H$ is a~cozero filter. Since all cozero sets of $\beta (X)$ form a~base of open sets in $\beta (X)$, there exists a~cozero (in $\beta (X)$) set $W\in\calN(\F)$ with $W\subseteq U$. Then $W\cap X$ is a~cozero set in $X$ which belongs to $\mathcal H$ and $W\cap X\subseteq U\cap X$. Hence $\mathcal H$ is a~cozero filter.
    Let us show that $\mathcal H$ is a~zero filter. Since $\beta (X)$ is Tychonoff, there exists a continuous function $f:\beta (X)\rightarrow [0,1]$ such that $f(\beta (X)\setminus U)=0$ and $f(\F)=1$. Then $f^{-1}([1/2,1])\subseteq U$ is a zero set that is contained in $\calN(\F)$. 
    Then $f^{-1}([1/2,1])\cap X\in \mathcal H$ is a zero set in $X$ that is contained in $U\cap X$. Hence $\mathcal H$ is a~zero filter. Since $\mathcal H$ is a closed-open ultrafilter and a zero-cozero filter, we get that $\mathcal H$ is a closed-open-zero-cozero ultrafilter. To derive a contradiction, assume that $\F\neq \mathcal H$. Then there exist disjoint zero subsets $A,B$ of $X$ such that $A\in \mathcal H$ and $B\in \F$. The set 
    $T=\{y\in\beta (X): X\setminus A\in y\}$ is an open neighborhood of $\mathcal F$ in $\beta (X)$. By the definition of $\mathcal H$, $T\cap X\in \mathcal H$, and thus $\emptyset=T\cap X\cap A\in \mathcal H$. The obtained contradiction implies $\F= \mathcal H$. 

 ($\Leftarrow$):  Let $\F$ be a zero-cozero-closed-open ultrafilter on a space~$X$ with no accumulation points. Then $\F$ belongs to $X^*$. 
   It is clear that the filter $\mathcal N(\F)$ traces on $X$ an open filter, which we denote by $\mathcal H$. We are going to show that $\mathcal H=\F$. Since $\F$ is an open ultrafilter, it suffices to check that $\F\subseteq \mathcal H$. Fix any $F\in \F$. Since $\F$ is a cozero ultrafilter, there exists a cozero set $F'\in \F$ such that $F'\subseteq F$. 
   Note that $X\setminus F'$ is a zero set and $X\setminus F'\notin \F$. Then the set $W=\{u\in\beta (X)\colon X\setminus F'\notin u\}$ is an open neighborhood of $\F$ in $\beta (X)$. Observe that $W\cap X\in\mathcal H$ and $W\cap X= F'\subseteq F$. Hence $\F\subseteq \mathcal H$, as required.
\end{proof}
\begin{corol}\label{cor2}
    Let $X$ be a Tychonoff space and $p\in X^*$. The~following are equivalent.
    \begin{enumerate}[\rm(a)]
    \item $p$ is a remote point;
    \item $p$ is a zero-cozero-closed-open ultrafilter with no accumulation points;
    \item $p$ is a zero-cozero-open ultrafilter with no accumulation points;
    \item $p$ is a zero-cozero-closed ultrafilter with no accumulation points.
    \end{enumerate} 
\end{corol}
\begin{proof}
    The equivalence of items (a) and (b) is established in Theorem~\ref{remote}. The equivalences (b) $\Leftrightarrow$ (c) and (b) $\Leftrightarrow$ (d) follow from items (1) and (2) of Corollary~\ref{corollary list}. 
\end{proof}




Proposition~\ref{pnormal} and Theorem~\ref{remote} imply the following.

\begin{proposition}\label{rem-ultra}
 Let $X$ be a normal space. Then $\F\in X^*$ is a remote point if and only if $\F$ is a closed-open ultrafilter on $X$ with no accumulation points.    
\end{proposition}

Recall that for a given space $X$ by $\NF$ we denote the ideal on $X$ consisting of all nowhere dense subsets of $X$ together with all finite subsets of $X$.

\begin{proposition}\label{rem_Baire}
A Tychonoff space $X$ has a remote point if and only if $X$ is not weakly $\NF$-compact.    
\end{proposition}

\begin{proof}
 ($\Rightarrow$) By Theorem~\ref{remote} there exists a closed-open-zero-cozero ultrafilter $\F$ on $X$ with no accumulation points. Since $\F$ has no accumulation points and $X$ is Tychonoff, for each finite subset $A$ of $X$ there exist a zero set $Z_A$ and a cozero set $V_A$ such that $A\subseteq Z_A\subseteq V_A$ and $V_A\notin \F$. Fix any infinite nowhere dense subset $B$ of $X$. Since $\F$ is an open ultrafilter, the open dense set $X\setminus \overline{B}$ belongs to $\F$. Since $\F$ is a zero-cozero filter, there exist a zero set $C\in\F$ and a cozero set $D\in\F$ such that $C\subseteq D\subseteq X\setminus\overline{B}$. Put $Z_B=X\setminus D$ and $V_B=X\setminus C$. Then $Z_B$ is a zero set, $V_B$ is a cozero set, $B\subseteq Z_B\subseteq V_B$ and $V_B\notin \F$. Thus for each $A\in\NF$ there exist a zero set $Z_A$ and a cozero set $V_A$ such that $A\subseteq Z_A\subseteq V_A$ and $V_A\notin \F$. Let $\phi: \NF\rightarrow cz(X)$ be defined by $\phi(A)=V_A$. Then for each finite subfamily $\mathcal A$ of $\NF$ the set $X\setminus \bigcup_{A\in\mathcal A}\phi(A)$ belongs to $\F$ and thus is nonempty. The function $\phi$ witnesses that $X$ is not weakly $\NF$-compact.

 ($\Leftarrow$) Since $X$ is not weakly $\NF$-compact, there exists a function $\phi: \NF\rightarrow cz(X)$ such that for each $A\in\NF$ there exists a zero set $Z_A$ satisfying $A\subseteq Z_A\subseteq \phi(A)$, and for each finite subfamily $\mathcal A\subseteq \NF$ the set $X\setminus \bigcup_{A\in\mathcal A}\phi(A)$ is nonempty. Consider the zero filter $\F$ generated by the family $$\{X\setminus  \bigcup_{A\in\mathcal A}\phi(A)\colon \mathcal A \hbox{ is a finite subset of } \NF\}.$$ 
By Zorn's lemma, $\F$ can be enlarged to a zero ultrafilter $\F'$. In order to show that $\F'$ is a remote point, fix any nowhere dense subset $A\subseteq X$. Since $X\setminus \phi(A)\in\F\subseteq \F'$ and $Z_A\subseteq \phi(A)$, we get that $Z_A\notin\F'$. Then $\{u\in \beta (X): Z_{A}\notin u\}$ is an open neighborhood of $\F'$ in $\beta (X)$ which is disjoint with $A$. Since the nowhere dense set $A$ was chosen arbitrarily, $\F'$ is a remote point.   
\end{proof}

\begin{proposition}\label{lweakly}
 A normal space $X$ is weakly $\NF$-compact if and only if $X$ is $\NF$-compact. 
\end{proposition}
\begin{proof}
The implication ($\Leftarrow$) is obvious. To prove the reverse implication, assume that a normal space $X$ is weakly $\NF$-compact. Consider any function $\phi: \NF\rightarrow o(X)$ such that $A\subseteq \phi(A)$ for all $A\in\NF$. By the Tietze-Urysohn Theorem~\cite[Theorem~2.1.8]{Eng}, for each closed $A\in \NF$ there exists a continuous function $f: X\rightarrow [0,1]$ such that $A\subseteq f^{-1}(0)$ and $X\setminus \phi(A)\subseteq f^{-1}(1)$. Put $Z_A=f^{-1}(0)$ and $V_A=f^{-1}([0,1/2))$. It is clear that $Z_A$ is a zero set, $V_A$ is a cozero set and $A\subseteq Z_A\subseteq V_A\subseteq \phi(A)$. Define a function $\psi: \NF\rightarrow cz(X)$ by $\psi(A)=V_{\overline{A}}$ for all $A\in\NF$. Since $X$ is weakly $\NF$-compact, there exists a finite family $\mathcal A\subseteq \NF$ such that $\bigcup_{A\in\mathcal A}\psi(A)=X$. Let $\mathcal B:=\{\overline{A}: A\in\mathcal A\}$. Taking into account that  $\psi(A)=\psi(\overline{A})\subseteq \phi(\overline{A})$ for each $A\in\NF$, we have that $\bigcup_{B\in\mathcal B}\phi(B)=X$. Hence $X$ is $\NF$-compact.  
\end{proof}

Propositions~\ref{rem_Baire} and~\ref{lweakly} imply the following.

\begin{proposition}\label{weaklyBaire}
A normal space $X$ has a remote point if and only if $X$ is not $\NF$-compact.    
\end{proposition}

\begin{proposition}\label{Bairemeasure}
A normal space $X$ is not $\NF$-compact if and only if there is an outer regular finitely additive measure~$\mu$ defined on $\B(X)$ such that $\mu(A)<\infty$ for all $A\in \NF$ and $\mu(X)=\infty$.   
\end{proposition}

\begin{proof}
($\Rightarrow$) By Proposition~\ref{weaklyBaire}, $X$ has a remote point. By Proposition~\ref{rem-ultra}, $X$ possesses a closed-open ultrafilter $\F$ with no accumulation points in $X$. Corollary~\ref{correspondence} implies that $\mu_\F$ is an outer regular $\{0,1\}$-measure that is defined on $\B(X)$ and vanishes on elements of $\NF$. Define a measure $\mu$ on $\B(X)$ as follows: for all $A\in \B(X)$, $\mu(A)=0$ if and only if $\mu_\F(A)=0$ and $\mu(A)=\infty$, otherwise. It is easy to see that $\mu$ is as required.

($\Leftarrow$) Fix any $A\in \NF$. Since the measure $\mu$ is outer regular and $\mu(A)<\infty$, there is an open set $U_A\supset A$ such that $\mu(U_A)<\infty$. Consider the function $\phi: \NF\rightarrow o(X)$ defined by $\phi(A)=U_A$. For any finite subfamily $\calA\subseteq\NF$ we have $$\mu(\bigcup_{A\in\calA}\phi(A))\leq \bigcup_{A\in\calA}\mu(U_A) <\infty.$$ Since $\mu(X)=\infty$, $\bigcup_{A\in\calA}\phi(A)\neq X$ for each finite $\calA\subseteq\NF$, as required.
\end{proof}

\section{Proofs of the main results}\label{5}
We are going to prove Theorem~\ref{main1}. 
We need to show that for a Tychonoff space $X$ the following conditions are equivalent:
\begin{enumerate}[\rm(a)]
    \item $X$ has a remote point;
    \item There exists a zero-cozero-closed-open ultrafilter on $X$ with no accumulation points;
    \item There exists a non-principal ultrafilter $u$ on $X$ such that $\uopen{u}=\uclosed{u}=\uzero{u}=\ucozero{u}$;
    \item $X$ is not weakly $\NF$-compact.
\end{enumerate}

\subsection*{Proof of Theorem~\ref{main1}}
The equivalence of items (a) and (b) is established in Proposition~\ref{remote},  whereas the equivalence of items (a) and (d) is established in Proposition~\ref{rem_Baire}. The implication (c)$\Rightarrow$(b) follows from Corollary~\ref{blabla} and Lemma~\ref{trivial1}. The implication (b)$\Rightarrow$(c) follows from Lemma~\ref{trivial2} and the fact that each ultrafilter, that extends a filter with no accumulation points, is non-principal. \qed   

\medskip

We are going to prove Theorem~\ref{wrB}. 
We need to show that for a normal space $X$ the following conditions are equivalent:
\begin{enumerate}[\rm(a)]
\item $X$ has a remote point;
\item There is a closed-open ultrafilter on $X$ with no accumulation points;
\item There exists a non-principal ultrafilter $u$ on $X$ such that $\uopen{u}=\uclosed{u}$;
\item $X$ is not $\NF$-compact;
\item There is an outer regular finitely additive $\{0,1\}$-measure~$\mu$ defined on $\B(X)$ such that $\mu(A)=0$ for all $A\in\NF$;
\item There is an outer regular finitely additive measure~$\mu$ defined on $\B(X)$ such that $\mu(X)=\infty$ and $\mu(A)<\infty$ for all $A\in\NF$.   
\end{enumerate}    

\subsection*{Proof of Theorem~\ref{wrB}}

The equivalence of items (a) and (b) is established in Proposition~\ref{rem-ultra}. The equivalence of items (a) and (d) is proven in Proposition~\ref{weaklyBaire}. The equivalence of items (b) and (e) follows from Corollary~\ref{correspondence}. The equivalence of items (d) and (f) is showed in Proposition~\ref{Bairemeasure}.  The implication (c)$\Rightarrow$(b) follows from Corollary~\ref{blabla} and Lemma~\ref{trivial1}. The implication (b)$\Rightarrow$(c) follows from Lemma~\ref{trivial2} and the fact that each ultrafilter, that extends a filter with no accumulation points, is non-principal.  \qed

\medskip

We are going to prove Theorem~\ref{Stone}.
We need to show that for a Tychonoff space $X$, $\rho(X)$ is homeomorphic to a subspace of $S(\B(X))$.
Moreover, if $X$ is additionally locally compact and $\sigma$-compact, then $\rho(X)$ is $\omega$-bounded.

\subsection*{Proof of Theorem~\ref{Stone}}
By Proposition~\ref{remote}, each remote point $p\in X^*$ is in fact a closed-open-zero-cozero ultrafilter with no accumulation points. It is easy to check that each closed-open ultrafilter on a space $X$ is an ultrafilter on the Boolean algebra $\B(X)$. So, each remote point $p\in X^*$ can be identified with the unique element $\phi(p)\in S(\B(X))$. Let us check that the defined above map $\phi\colon \rho(X)\rightarrow S(\B(X))$ is a topological embedding. It is clear that $\phi$ is injective. Fix a remote point $p\in X^*$ and an open neighborhood $W$ of $\phi(p)$. Without loss of generality we can assume that $W=\{u\in S(\B(X)): A\in u\}$, for some $A\in \B(X)$. Since $\phi(p)\in W$, we get $A\in \phi(p)$, which implies that $A\in p$. Since $p$ is a cozero filter, there exists a cozero set $B\in p$ such that $B\subseteq A$. Then $U:=\{y\in\rho(X): X\setminus B\notin y\}$ is an open neighborhood of $p$ such that $\phi(U)\subseteq W$. Hence $\phi$ is continuous. It remains to show that the map $\phi$ is open onto its image. Fix a zero subset $Z$ of $X$ and a basic open set $V=\{y\in \rho(X): Z\notin y\}$ in $\rho(X)$. It is a routine to check that $$\phi(V)=\{u\in S(\B(X)): X\setminus Z\in u\}\cap \phi(\rho(X)).$$ Hence $\phi$ is open onto its image, as required. 

Assume that the space $X$ is locally compact and $\sigma$-compact. Then there exists a family $\{V_n:n\in\w\}$ of open subsets of $X$ such that $\bigcup_{n\in\w}V_n=X$, $\cl_X(V_n)$ is compact and $\overline{V_n}\subseteq V_{n+1}$ for each $n\in\w$. 
If $\rho(X)=\emptyset$, then there is nothing to prove. Assume that $\rho(X)\neq \emptyset$. Note that $\sigma$-compactness of $X$ implies that $X$ is normal. So, elements of $\rho(X)$ are closed-open ultrafilters on $X$ with no accumulation points, and the topology on $\beta(X)$ is generated by the sets $\{z\in\beta(X): O\in z\}$, where $O$ is an open subset of $X$.  Fix a countable subset $C=\{u_n:n\in\w\}\subseteq \rho(X)$. We are going to show that $\cl_{\rho(X)}(C)=\cl_{\beta(X)}C$, which in turn would imply that $\rho(X)$ is $\omega$-bounded. 
It is clear that $\cl_{\rho(X)}(C)\subseteq \cl_{\beta(X)}C$. To show the converse inclusion, fix any $y\in \cl_{\beta(X)}(C)$ and any nowhere dense subset $A$ of $X$. It suffices to show that $y\notin\cl_{\beta(X)}(A)$. Without loss of generality we can assume that $A$ is closed in $X$, as otherwise we can substitute $A$ with $\cl_X(A)$. Since $C\subseteq \rho(X)$, for each $i\in\w$ there exists an open set $U_i\in u_i$ such that $\cl_X(U_i)\cap A=\emptyset$. Let $W_i:=U_i\setminus \cl_X(V_i)$. Since $\cl_X(V_i)$ is compact, $W_i\in u_i$ for all $i\in\w$. For each $x\in X$ there exists $i(x)\in\w$ such that $V_{i(x)}$ is an open neighborhood of $x$. Thus the family $W:=\{W_i:i\in\w\}$ is locally finite in $X$.  It follows that $$\cl_X(\bigcup W)=\bigcup_{n\in\w}\cl_X(W_n)\subseteq \bigcup_{n\in\w}\cl_X(U_n)\subseteq X\setminus A.$$  
We claim that $\cl_X(\bigcup W)\in y$. Otherwise, since $y$ is a closed ultrafilter, $X\setminus \cl_X(\bigcup W)\in y$. But then $\{z\in\beta(X): X\setminus \cl_X(\bigcup W)\in z\}$ is an open neighborhood of $y$ disjoint with $C$, which contradicts the assumption. Since $\cl_X(\bigcup W)\subseteq X\setminus A$, the set $\{z\in\beta(X): X\setminus A\in z\}$ is an open neighborhood of $y$ disjoint with $A$. Hence $y\notin\cl_{\beta(X)}(A)$, as required.  \qed

%



\end{document}